\documentclass{amsart}

\usepackage{amsfonts}
\usepackage{latexsym}
\usepackage{amscd,amsmath,amsthm,amssymb}

\newcommand{\R}{\mathbb R}

\newcommand{\HH}{{\mathbb{H}}}

\def\a'{\`a}
\def\e'{\`e}
\def\o'{\`o}
\def\u'{\`u}

\newtheorem{teor}{Theorem}[section]
\newtheorem{lemma}[teor]{Lemma}
\newtheorem{prop}[teor]{Proposition}
\newtheorem{defi}[teor]{Definition}
\newtheorem{remk}[teor]{Remark}
\theoremstyle{definition}
\newtheorem*{Acknow}{Acknowledgments}

\begin{document}

 \title{Subspaces of a para-quaternionic Hermitian vector
space}
\thanks{Work done under the programs of GNSAGA-INDAM of C.N.R. and PRIN07 "Riemannian metrics and differentiable structures" of MIUR (italy)}

\maketitle

\centerline{MASSIMO VACCARO}
\vskip 0.2cm
\centerline{\em{Dipartimento dell'Ingegneria di Informazione e Matematica Applicata}}
\centerline{ \em{Universit\a' di Salerno, 84084 - Fisciano (SA) , Italy}}
\centerline{\em{massimo\_vaccaro@libero.it}}


\vskip 0.2cm

\textbf{Abstract:} Let $(\tilde Q,g) $ be a para-quaternionic Hermitian structure on the real vector space $V$.
By referring to the tensorial presentation  $(V, \tilde{Q},g) \simeq
(H^2 \otimes  E^{2n}, \mathfrak{sl}(H),\omega^H \otimes \omega^E)$,
we give an explicit description, from an affine and metric point of view,
of main classes of subspaces of
$V$ which are invariantly defined  with respect to the structure group of $\tilde{Q}$  and $(\tilde{Q},g)$ respectively.

\vskip 0.2cm

{\em keywords:} Para-quaternionic Hermitian structure,  complex subspaces, para-complex subspaces, totally real subspaces.

\vskip 0.1cm


\section{Introduction} In the last years, quaternionic-like structures have captured an  increasing interest both in  mathematics and
physics. In particular,  many  recent researches have focused on  para-quaternionic structures which are the object of this article
from a basic geometrical point of view.

Para-geometries play an important role in physics in some supersymmetric theories.
For instance in \cite{CMMS1}, \cite{CMMS2}, it was shown   that the target space for scalar fields in 4-dimensional
Euclidean N = 2 supersymmetry carries a special para-Kaehler structure similar to
the special Kaehler structure which arises on the target space of scalar fields for
N = 2 Lorentzian 4-dimensional supersymmetry.
Also,  in \cite{M}, where   the role of special geometry in the theory of supersymmetric black holes is explained,
the  target metric is   (Riemannian) quaternionic Kaehler or (neutral) para-quaternionic Kaehler according if  the space-time signature of the metric is Lorentzian  or Euclidean respectively.

In this article,  which is the first part of a research project about submanifolds of a para-quaternionic Kaehler manifold,
we deal with special subspaces of a para-quaternionic Hermitian vector space. A brief description of the results  is given below.


Let $V$ be a real vector space  endowed
with a para-quaternionic structure $\tilde Q \subset End(V)$, i.e.
$\tilde Q$ is isomorphic to $Im \tilde \HH$ where $\tilde \HH$ is the  Clifford algebra of para-quaternions.
It is known that  $\dim V= 2n$ and one has an isomorphism $(V, \tilde Q) \simeq (H^2 \otimes E^n,
\mathfrak{sl}(H))$, where $H$ and $E$ are real vector spaces and
 $\mathfrak{sl}(H)$ is the Lie algebra of the
special linear group $SL(H)$.

Here we consider  a  \textit{para-quaternionic Hermitian}  vector space $(V, \tilde{Q},g)$,
where $g$ ia a  $\tilde Q$-hermitian metric on $V$. In this case
the compatibility conditions imply that $\dim V=4n$ and that  $g$ is  pseudo-Euclidean  of (neutral)  signature
$(2n,2n)$. Moreover one has an isomorphism  $(V^{4n}, \tilde{Q},g)
\simeq (H^2 \otimes  E^{2n}, \mathfrak{sl}(H)),\omega^H \otimes
\omega^E)$ where $\omega^H$ and $\omega^E$ are two symplectic
forms on $H$ and $E$ respectively.

 The para-quaternionic Hermitian structure  naturally defines some  classes of \textit{special subspaces} of $V$ in terms of their
  behaviour with respect
 to the endomorphisms of  $\tilde Q$ and to the metric $g$, which are interesting to consider:
 \textbf{para-quaternionic}, \textbf{complex}, \textbf{totally
complex}, \textbf{weakly para-complex}, \textbf{totally
para-complex}, \textbf{nilpotent},  \textbf{real}, \textbf{totally real}.

Also by referring to the tensorial presentation of a para-quaternionic Hermitian vector space
 there are some  classes of subspaces of $V$ which is natural to consider: first of all the \textbf{product} subspaces  and among them the \textbf{decomposable} subspaces,
 furthermore   some   $\bf {U^{F,T}}$
subspaces which we defined as depending on a symplectic basis of $H$, on a subspace $F \subset E$ and on a linear map $T:F \rightarrow E$  (see def.(\ref{spazi U^{F,T}})).
Indeed a generic  subspace in $V$ is not a ${U^{F,T}}$ subspace, but it turns out that any $U \subset V$
admits a decomposition into a pair of such subspaces (prop. \ref{decomposition of a subspace}).

The main purpose of this article is to  give an explicit  description of the
special subspaces of the para-quaternionic Hermitian space $(V, \tilde Q,g)$ in terms of the tensorial presentation $(H^2 \otimes  E^{2n}, \mathfrak{sl}(H)),\omega^H \otimes
\omega^E)$.
After proving that the \textit{para-quaternionic subspaces coincide with the products $H \otimes E', \; E' \subseteq E$} (prop. \ref{para-quaternionic subspaces proposition}), the  basic  tool consists in  restricting to \textbf{pure} subspaces,
 not containing any non trivial  para-quaternionic subspace, and by showing that \textit{pure special subspaces are $U^{F,T}$ subspaces}.
Viceversa we also give the precise \textit{conditions for a $U^{F,T}$ subspace to be a special subspace of any given type}.

 This presentation is also useful from the metrical point of view to determine, for each  subspace, the signature of the
induced metric. We  give then the conditions for the above special
subspaces to be  $g$-non degenerate.

Finally we notice that the results obtained with regard to the tensorial presentation
of the geometry of a para-quaternionic
Hermitian vector space can be extended to the quaternionic case due to the fact that
the complexification of a quaternionic Hermitian vector space $V^{4n}$ has a natural
identification with the tensor product $H \otimes E$ of two complex vector spaces
of dimension 2 and 2n respectively (\cite{S1},\cite{S2}).

This work develops and completes  a research undertaken during  the Ph.D. thesis whose  advisor was professor Dmitri Alekseevsky.

\section{Preliminaries}


\begin{defi}
Let $V$ be a real vector space of dimension $n$ and  $K \in End
(V)$ such that $K^2=Id$. Let denote $V_K^+$ and $V_K^-$  the +1 and -1
eigenspaces of $K$. Then $K$ is called a \textbf{product structure
on $V$} if $\dim V_K^+,\, \dim V_K^- > 0$. A \textbf{para-complex
structure on $V$} is a product structure with $\dim V_K^+ = \dim V_K^-$.

 A triple
$(J_1,J_2,J_3)$ of anticommuting endomorphisms of $V$ satisfying
the relations:
\begin{equation} \label{eight}
-J_1^2=J_2^2= J_3^2 =Id , \quad J_1 J_2 = J_3
\end{equation}
 is called a \textbf{para-hypercomplex structure} on $V$ \footnote{
Observe that  $J_2$ and $J_3$ are para-complex structures on $V$. In
fact, since the complex structure $J_1$   anti-commute with $J_2\,$, $J_1 (V_{J_2}^+) \subseteq V_{J_2}^-$
and $J_1 (V_{J_2}^-) \subseteq V_{J_2}^+$, which implies $\dim V_{J_2}^+ = \dim
V_{J_2}^-$, and analogously for $J_3$.}. A Lie subalgebra $\widetilde{Q} \subset \mathfrak{gl}(V)$
is called a \textbf{para-quaternionic structure} on $V$ if it admits
 a basis $(J_1,J_2, J_3)$ satisfying the relations
(\ref{eight}). Such a para-hypercomplex structure
is called  an admissible basis of
$\widetilde{Q}$.
\end{defi}

 A para-hypercomplex structure $(J_1,J_2,J_3)$
defines on $V$ the structure of a left module over the real algebra $\widetilde\HH$ of para-quaternions   which is the real
algebra generated by unity 1 and $i, j, k$ satisfying
\begin{equation} \label{one}
-i^2 = j^2 = k^2= 1, \quad ij = -ji = k.
\end{equation}

$\widetilde\HH$ is  isomorphic, as real pseudo-normed
algebra, to the algebra $ Mat_2(\R)$ of real $(2 \times
2)$-matrices, the isomorphism being given by
\begin{equation} \label{isomorphism H with 2
x 2 matrices} \Phi : \textbf{q}=q_0 + q_1 i  + q_2 j+ q_3
k \mapsto \left( \begin{array}{cc}  q_0 - q_3
& q_2 - q_1\\
q_2 + q_1 & q_0 + q_3
\end{array}
 \right)
\end{equation} where $\mathcal{N}(q):=q \bar{q}= q_0^2 +q_1^2-q_2^2-q_3^2= det(\Phi(\textbf{q}))$.

 The \textbf{standard
para-hypercomplex structure $(\mathcal{I}, \mathcal{J},
\mathcal{K})$ of $V= \R^2$} is represented, in the canonical basis, by
\begin{equation} \label{standard para-hypercomplex structure of R^2}
\mathcal{I}=\left(
\begin{array}{cc}
0 & -1 \\ 1 & 0
\end{array}
\right), \qquad \mathcal{J}=\left(
\begin{array}{cc}
0 & 1 \\ 1 & 0
\end{array}
\right), \qquad \mathcal{K}=\left(
\begin{array}{cc}
-1 & 0 \\ 0 & 1
\end{array}
\right).
\end{equation}
Observe that $<\mathcal{I}, \mathcal{J}, \mathcal{K}>_\R \simeq
\mathfrak{sl}_2(\R)$ the matrix Lie algebra of zero
trace real $(2 \times 2)$-matrices  of the unimodular Lie group $SL_2(R)$.

Generalizing, we define the \textbf{standard para-hypercomplex
structure $\widetilde{H}=(I,J,K)$ of $R^{2n}= \R^2 \oplus \ldots \oplus \R^2$} represented, in the
canonical basis, by
\begin{equation} \label{standard para-hypercomplex structure of R^2n}
I=\mathcal{I} \oplus \mathcal{I} \oplus \ldots \oplus \mathcal{I};
\quad J=\mathcal{J} \oplus \mathcal{J} \oplus \ldots \oplus
\mathcal{J}; \quad K=\mathcal{K} \oplus \mathcal{K} \oplus \ldots
\oplus \mathcal{K}.
\end{equation}

By the  identification $\widetilde\HH \cong Mat_2(\R)$ of
(\ref{isomorphism H with 2 x 2 matrices})  and from Wedderburn
theorem, stating that every representation of a unitary,
associative, semisimple algebras  is direct sum of standard
representations, it results the following
\begin{prop} \label{unique irreducible 2-dimensional H-module}
$ $
\begin{itemize}
\item There exists a unique, up to isomorphism,
irreducible $\widetilde\HH$-module $H^2 \simeq \R^2$.
\item Every
$\widetilde\HH$-module $V^{2n}$ is reducible as a direct sum
 $V=H^2 \oplus  \ldots \oplus H^2$.
\end{itemize}
\end{prop}

Note that to have a direct sum decomposition of the $\widetilde\HH$-module
$(V^{2n},(I,J,K))$,  into invariant 2-dimensional
subspaces $U_1, \ldots, U_n$, one considers a basis ${e_i^+}$ of
$V_J^+$, eigenspace of the para-complex structure $J$ associated
to the eigenvalue 1 (then $Ke_1^+, \ldots, Ke_n^+$ is a basis of $V_J^-$).
The 2-dimensional subspaces
 \begin{equation} \label{invariant 2-dimensional subspace in pqH vector space}
 U_i=<e_i^+, Ke_i^+>, \qquad i=1,\ldots,n
 \end{equation}
 are clearly $\widetilde{H}$-invariant, irreducible and isomorphic as
 $\widetilde\HH$-modules.
Choosing the  basis $<e_i^+ - Ke_i^+, e_i^+ + Ke_i^+ >$ in each
$U_i$, $\widetilde{H}$
 corresponds to the standard para-hypercomplex structure of
$R^{2n}$ given in (\ref{standard para-hypercomplex structure of
R^2n}).

\vskip 0.3cm

Let $H^2$ and $E^n$ be real vector spaces. For any fixed basis $(h_1,
h_2)$ of $H$, one has the identification $H \simeq \R^2$: we define a corresponding \textbf{standard para-hypercomplex
structure on $H^2 \otimes E^n$}  by
\begin{equation} \label{standard para_hypercomplex structure}
I=I(h \otimes e)=\mathcal{I} h \otimes e, \quad J=J(h \otimes e)=\mathcal{J} h \otimes e,
\quad K=K(h \otimes e)=\mathcal{K} h \otimes e
\end{equation}
 with $\mathcal{I},\mathcal{J},\mathcal{K}$ given in (\ref{standard para-hypercomplex structure of R^2}) and the
\textbf{standard para-quaternionic structure $\mathfrak{sl}_2(\R)
\otimes Id$ on $H^2 \otimes E^n$} generated by any standard
para-hypercomplex structure.

Since  $\mathfrak{sl}_2(\R) \simeq \mathfrak{sl}(H)$,
the Lie algebra of the  Lie group $SL(H)$ of unimodular transformations of $H$,
 we will use the equivalent notations
\[\mathfrak{sl}_2(\R) \otimes Id \simeq \mathfrak{sl}_2(\R)  \simeq \mathfrak{sl}(H).\]

From prop. (\ref{unique irreducible 2-dimensional H-module}) it follows immediately
\begin{prop} \label{isomorphism between para-quaternionic space and tensor product H x E and sl(H)}
Any vector space $V^{2n}$ with a para-hypercomplex structure $\tilde
H$ is isomorphic to $H^2 \otimes E^n$ with a standard
para-hypercomplex structure.  Consequently any para-quaternionic
vector space $(V^{2n},\tilde Q)$ is isomorphic to $(H^2 \otimes
E^n, \mathfrak{sl}(H))$.
\end{prop}

\vskip 0.2cm

\begin{defi}
Let $(V,K)$ be a $2n$-dimensional para-complex vector space. A
pseudo-Euclidean scalar product $g=<\cdot,\cdot>$ on $(V,K)$  is
called $K-$\textbf{Hermitian}  if $K$ is a skew-symmetric
endomorphism of \mbox{$(V,<\cdot,\cdot>$)}.

A vector space $V$ endowed with a para-complex structure $K$
 and a $K$-Hermitian scalar product $g$ is called a \textbf{para-Hermitian vector space
$(V, K, g)$}\footnote{The reason why we do not consider $n$-dimensional  vector spaces endowed with a product structure not para-complex is that
the  metric on such spaces, direct sum of a pair of totally isotropic eigenspaces of different dimensions,  is always degenerate.}.
\label{footnote about metric on product structure subspaces}

 A para-hypercomplex structure $(J_1,J_2,J_3)$
 on $V$ is called \textbf{para-hypercomplex Hermitian structure} with respect to the pseudo-Euclidean scalar product $g$
 if its endomorphisms are skew-symmetric with respect to $g$.

 A para-quaternionic structure $\widetilde{Q}$ on $V$ is called
a \textbf{para-quaternionic Hermitian structure} with respect to $g$
if some (and hence any) admissible basis is Hermitian with respect
to $g$.
\end{defi}

The eigenspaces associated to any para-complex structure are clearly
totally isotropic, then  a \textit{para-hermitian metric has   neutral signature}
and this leads to the 

\begin{prop} \label{dimension 4n of a Hermitian para-quaternionic vector space}
The dimension of a vector space $V^{2n}$, endowed with a
para-hypercomplex (resp. para-quaternionic) Hermitian
 structure $(\widetilde{H},g)$  (resp.$(\widetilde{Q},g)$), is a multiple of 4.
\end{prop}

\vskip 0.2cm

Let consider now  the standard para-quaternionic vector space
$(H^2 \otimes E^{2n}, \mathfrak{sl}(H))$. Let $\omega^E$ be a
symplectic form on $E$ and $\omega^H=h_1^* \, \wedge \, h_2^*$ a volume form on $H$.


\textit{The 2-form $g_0=\omega^H \otimes \omega^E$ is a pseudo-Euclidean  $\tilde Q$-Hermitian metric on  $H^2 \otimes E^{2n}$}.
It is trivial to verify that $g_0$ is bilinear, symmetric and non degenerate. Moreover,
computing  on decomposable
vectors,  $\forall A \in \mathfrak{sl}(H)$ one has
 \[g(Ah' \otimes e, \tilde{h} \otimes e')=
\omega^H(Ah',\tilde{h})
 \omega^E(e,e')= -\omega^H(h',A\tilde{h})
 \omega^E(e,e')= -g(h' \otimes e, A \tilde{h} \otimes e').\]

\begin{defi}
The $4n$-dimensional space
\mbox{$(H^2 \otimes E^{2n}, \mathfrak{sl}(H),\omega^H \otimes
\omega^E)$}
is a \textbf{standard para-quaternionic Hermitian space}.\\
\end{defi}

\begin{prop}
Let $V^{4n}$ be a vector space with a para-quaternionic Hermitian
structure  $(\widetilde{Q},g)$. Then  the para-quaternionic
Hermitian space   $(V,\widetilde{Q},g)$ is isomorphic to a
standard para-quaternionic Hermitian space.
\end{prop}

\begin{proof}
By proposition (\ref{isomorphism between para-quaternionic space and tensor
product H x E and sl(H)}) we identify
$(V^{4n},\widetilde{Q}) \simeq (H^2 \otimes E^{2n},
\mathfrak{sl}(H))$.

Then the  given para-quaternionic Hermitian metric $g$ on  $H^2 \otimes E^{2n}$
can be written   $g=\omega^H \otimes \omega^E$ where
 $\omega^H= h_1^* \wedge h_2^*$ is the standard volume form on $H$ and  $\omega^E$ is defined by
\[ \omega^E(e,e'):=
\frac{g(h \otimes e, h' \otimes e')}{\omega^H(h,h')},
\]
for one (and hence any) pair of linearly independent vectors $h,h'$. It is straightforward to prove that the right member is well defined by observing that, by Hermitianicity, decomposable vectors are always isotropic and $g(h_1 \otimes e, h_2 \otimes e')+g(h_2 \otimes e, h_1 \otimes e')=0$. 
Moreover $\omega^E$ is clearly symplectic.
\end{proof}
Finally, let make the following remark that we will use in next section.
As  an  $\widetilde\HH$-module,  on a  para-hypercomplex
Hermitian vector space $(V^{4n},\{I,J,K\},g)$  we define the
($\widetilde\HH$-valued)-Hermitian  product $(\cdot)=(\cdot)_{\{I,J,K\}}$  by:
\begin{equation} \label{Hermitian scalar product on a PQH vector space}
\begin{array}{llll}
( \cdot ) : & V \times V & \rightarrow &
\tilde \HH \\
 & (X,Y) & \mapsto &  X \cdot Y = g(X,Y) + i g(X,IY)- j g(X,JY)- k g(X,KY). \\
\end{array}
\end{equation}

When considering a para-quaternionic Hermitian vector space $(V,\widetilde{Q},g)$ , we observe that the
($\widetilde\HH$-valued)-Hermitian  product defined in (\ref{Hermitian scalar product on a PQH vector space})
depends on the chosen admissible basis $\{I,J,K\} \in \widetilde{Q}$. Two Hermitian products
$( \; \cdot \; )_{\{I,J,K\}}$, $ ( \; \cdot \; )_{\{I',J',K'\}}$,
 referred to different admissible basis, are related by an inner automorphism of $\widetilde\HH$. This implies that
\[\mathcal{N}(Im (X \cdot Y)_{\{I,J,K\}})=\mathcal{N}(Im (X \cdot Y)_{\{I',J',K'\}}), \quad \forall X,Y \in V\]
since the  real part of the norm $\mathcal{N}((X \cdot Y))$ is independent on the  basis $\{I,J,K\}$.

\section{Subspaces of a para-quaternionic Hermitian
vector space}
 From now on, the para-quaternionic Hermitian vector space
 $(V^{4n},\widetilde{Q},g)$  will be the standard para-quaternionic Hermitian vector space
 $(H^2 \otimes E^{2n}, \mathfrak{sl}(H), \omega^H \otimes
\omega^E)$.  We recall that, in this case, any para-hypercomplex admissible basis $(I,J,K)$ of
$\widetilde{Q}=\mathfrak{sl}(H)$ is a  standard para-hypercomplex Hermitian  structure which
 corresponds to a symplectic basis
  $(h_1, h_2)$ of $H$ such that
$\omega^H = h_1^* \wedge h_2^*$ and one has
\begin{equation} \label{standard hypercomplex structure}
\begin{array}{ll}
 I(h_1 \otimes e) = h_2 \otimes e; \qquad & I (h_2 \otimes e)
= -h_1 \otimes e;\\
 J (h_1 \otimes e) = h_2 \otimes e; \qquad & J (h_2 \otimes e)
= h_1 \otimes e;\\
 K (h_1 \otimes e) = -h_1 \otimes e; \qquad & K (h_2 \otimes e)
= h_2 \otimes e.\\
\end{array}
\end{equation}

If $A$ and $B$ are subspaces in $E$,  in the following we will denote
$\omega^E(A \times B)$  the restriction of the symplectic
form $\omega^E$ to the subspace $A \times B$ of $E \times E$.
 Moreover, by
saying that  $\omega^E(A \times B)$ is degenerate, we will mean
that  there exists $0 \neq b_0 \in B$ such that $\omega^E(a,b_0)=0, \,
\forall a \in A$ i.e. \[\ker  \, \omega^E(A \times B)= \{b \in B
\; | \; \omega^E(a,b)= 0, \; \forall a \in A \} \neq \{0\}.\] In
the following we will give definitions and explicit descriptions
of some relevant classes of subspaces in $V$.  The definitions of the subspaces that we consider involve only the para-quaternionic structure $\widetilde Q$,
except for totally complex, totally para-complex and totally  real subspaces for which we  take into account also the Hermitian metric $g$.


\subsection{Special subspaces of $\bf V=H \otimes E$.}

Fixed a symplectic basis $(h_1,h_2)$ of $H$, any $X \in H \otimes E$ can be written in a unique way
\begin{equation}
 X= h_1 \otimes e + h_2 \otimes e', \qquad  e,e' \in E.
\end{equation}
Let denote by $p_i: H \otimes E \rightarrow E, \; i=1,2$, the natural linear projections defined by
\begin{equation} \label{p1 and p2 maps}
 p_1(X)=e; \quad p_2(X)=e'.
\end{equation}
If $U$ is a subspace then  $p_1(U)=E_1$, $p_2(U)=E_2$ are subspaces of $E$, depending  on the chosen symplectic basis $(h_1,h_2)$ in $H$.
Notice that the sum $p_1(U) + p_2(U)$ is invariant.



With respect  to the tensor product  structure, the following subspaces of $V$ can
 be defined. First of all there are the \textbf{product subspaces}
 $H' \otimes E'$, with $H' \subseteq H$ and $E' \subseteq E$ any given subspaces.
  Referring to the dimension of the non trivial  factor in $H$, only two classes of
  such subspaces are to be considered.

A non zero product subspace $U=h \otimes E' \subset H \otimes E$ where $h$ is a fixed  element in $H$
and $ E' \subset E$ a subspace, will be called a \textbf{decomposable subspace} (meaning that all
elements in $U$ are decomposable vectors; subspaces $H \otimes e', e' \in E$  will not be considered under such
terminology). W.r.t. the  metric $g$, any decomposable subspace is totally isotropic.



We introduce another  important family of subspaces that we denote by $U^{F,T}$.
\begin{defi}\label{spazi U^{F,T}}
Let  $(h_1,h_2)$ be a symplectic basis of $H$,  $F \subseteq E$ a
subspace, $T: F \rightarrow E$ a linear map. We define the
following subspace of $H \otimes E$
\begin{equation} \label{U_FT form}
U^{F,T}:= \{h_1 \otimes f + h_2 \otimes Tf, \; f \in F\}.
\end{equation}
Note that the map
\begin{equation} \label{isomorphism U_F}
\begin{array}{lll}
\phi: & F \rightarrow  & U^{F,T} \\
  & f  \mapsto & h_1 \otimes f + h_2 \otimes Tf
\end{array}
\end{equation}
is an isomorphism of real vector spaces.
By saying that a subspace $U \subset H \otimes E$ is a \textbf{$ \bf U^{F,T}$ subspace}, we will mean that   it admits the form
 (\ref{U_FT form}) w.r.t. some symplectic basis $(h_1,h_2)$ of $H$.
\end{defi}

As a first  example  of subspaces  admitting the $U^{F,T}$ form we
have the decom\-po\-sa\-ble subspaces $U=h \otimes E', \; h \in H, \; E'
\subseteq E$: in any basis $(h_1=h,h_2)$ let $F= E'$ and $T\equiv
0$. Also, in a basis $(h_1,h_2)$  with $h_1,h_2 \neq h$, one has
$F=p_1(U)=p_2(U)=E'$ and $T= \lambda Id$ where $\lambda=
\frac{\beta}{\alpha}$ for $h = \alpha h_1 + \beta h_2$. On the
other hand,  $U$ does not admit the form (\ref{spazi U^{F,T}})
w.r.t. any basis $(h_1,h_2\equiv  \alpha h), \; \alpha \in \R$.

It is immediate to prove the following
\begin{prop} \label{spazi U(F,L)}
$ $
\begin{itemize}
\item[a)]     A subspace $U$ is a $U^{F,T}$ subspace  iff there exists $h \neq 0$ in $H$  such that $(h \otimes E) \cap U= \{ 0\}$.
\item[b)]   W.r.t. the symplectic basis  $(h_1,h_2)$, the map $T$  for the subspace $U=U^{F,T}$ is injective iff  $(h_i \otimes E) \cap U= \{ 0\}, \; i= 1,2$.
\end{itemize}
\end{prop}


Observe that if $U = U^{F,T}$ w.r.t. $h_1,h_2$ and also  $U = U^{F',T'}$
 w.r.t. $h'_1,h'_2$ where
 \[h_1 = \alpha h_1' + \beta h_2', \qquad \mbox{and} \qquad h_2 = \gamma h_1' + \delta h_2' ,\]
 then \[F' = (\alpha Id  + \gamma T)F \qquad \mbox{ and} \qquad   T' = (\alpha Id  + \gamma T)^{-1}(\beta Id + \delta T).\]

\begin{prop} \label{number of decomposable ina a U_FT subspace}
A $U^{F,T}$ subspace can always be written as $U^{F',T'}$ with $T'$ injective by performing a suitable change of basis in $H$.
\end{prop}

\begin{proof}
A subspace $U=U^{F,T}= \{h_1 \otimes f + h_2 \otimes Tf, \; f \in F\}$ of dimension $m$ contains at most $m$ distinct non zero
decomposable vectors  $k_i \otimes f_i, \; i= 1, \ldots,t$,  if the $k_i \in H$ are pairwise independent.
In fact,  $ k_i \otimes f_i \in U, \; i=1, \ldots,t$, with $k_i=a_i h_1+b_i h_2$, iff  $f_i \in F$ and $Tf_i =  b_i / a_i  \; f_i= \lambda_i f_i$.
Considering the restriction of $T$ to the subspace $<f_1, \ldots, f_t>$ the conclusion follows.
\end{proof}

Remark that, from the isomorphism   (\ref{isomorphism U_F}), the decomposable subspaces contained in a $U=U^{F.T}$ subspace are direct addends in $U$.


In general, a subspace $U \subset V$ does not admit the form $U^{F,T}$: an example is given by any product subspace $H \otimes E', \; E' \subset E$.
On the other  hand, any subspace can be written as direct sum of some $U^{F,T}$ subspaces. In fact,  we have the following
\begin{prop} \label{decomposition of a subspace}
Any subspace $U$ can be written in the forms
\begin{enumerate}
\item  $(h \otimes F') \oplus U^{F'',T''}$ for some $h \in H$
and  $U^{F'',T''}$ of maximal dimension w.r.t. all subspaces of the form $U^{F,T}$ contained in $U$.

\item $k_1 \otimes F_1 \oplus \ldots \oplus k_s \otimes F_s \oplus U^{\widetilde F,\widetilde T}$
with the $k_i \in H, i=1,\dots,s,$ pairwise independent and  $U^{\widetilde F,\widetilde T}$ of maximal dimension w.r.t. all subspaces of the form $U^{F,T}$ containing  no  decomposable subspace.

\end{enumerate}
\end{prop}

\begin{proof}
1) If $U=U^{F,T}$ there is nothing to prove. Otherwise let consider all maximal decomposable subspaces in $U$ and let
 $U_1=h \otimes F' \subset U$ of minimum dimension among them.
Then any subspace complementary to $U_1$ in $U$ is  clearly  a $U^{F,T}$ subspace and of maximal dimension.

\vskip 0.3cm

2) If $U$ contains no decomposable subspaces, there is nothing to prove; otherwise let $U_1\subset U$ be of maximal dimension among all $U^{F,T}$  subspaces in $U$  containing no
decomposable subspaces. Then  $U=U_1 \oplus U_2$ with $U_2$ any complementary.
The subspace $U_2$ can be decomposed into a direct sum of some maximal decomposable subspaces and a $U_3=U^{F,T}$ subspace containing no
decomposable subspaces i.e $U_2= \bigoplus_{i=1}^s (k_i \otimes F_i) \oplus U_3$ with the $k_i$ pairwise independent. Necessarily $U_1 \oplus U_3$ contains some maximal decomposable subspaces,
then $U= U_1  \oplus  \bigoplus_{i=1}^s (k_i \otimes F_i) \oplus  \bigoplus_{j=s+1}^t (k_j \otimes F_j) \oplus U_4$ with $U_4=U^{F,T}$ a subspace containing no
decomposable subspaces  with $\dim U_4 < \dim U_3$ and the $k_i, i=1, \ldots,t$ pairwise independent. By carrying on such procedure, the thesis follows.
\end{proof}

Concerning the unicity of the presentation of the
form $U^{F,T}$ we state the following lemma whose proof is
straightforward (see proof of prop. \ref{number of decomposable ina a U_FT subspace})
\begin{lemma}  \label{subspaces with no decomposable}
Given a subspace $U \subset H \otimes E$ the following conditions are equivalent:
1)  $U$ does not contain any non zero decomposable vectors; \\
2) $U=U^{F,T}$  w.r.t. \textbf{any}
symplectic basis  $\mathcal B=(h_1,h_2)$,  $F= F(\mathcal B) \subset E$ and  $T= T(\mathcal B)$ injective.\\
3) there exists a basis $(h_1,h_2)$, such that  $U=U^{F,T}$  for
some subspace $F \subset E$ and some linear injective map $T$ with no invariant line
(i. e. if $Tf= \lambda f, \; \lambda \in \R, \; f \in F$ then
$f=0$).

A necessary condition for 1),2),3) to hold is that $\dim U \le \dim E$.
\end{lemma}

From the metrical point of view we have easily the following
\begin{lemma}
Let $U=U^{F,T}$ be a subspace and $\phi$ the isomorphism  (\ref{isomorphism U_F}).
Let \\$g_F= \phi^* g_U$ be the  pullback of  the (possibly degenerate)  restriction of  $g$ to $U$. Then
\begin{equation} \label{Hermitian metric on F}
 g_F(f,f')= -2 (\omega^E \circ T)^{sym}(f,f')=   -[ \omega^E(Tf,f') + \omega^E(Tf',f)]
\end{equation}
\end{lemma}




\subsection{Para-quaternionic subspaces} \label{section: Para-quaternionic subspaces}
\begin{defi} A   subspace $U \subset V$ is called
\textbf{para-quaternionic}  if it is $\widetilde{Q}$-invariant,
or equivalently,
 for one and hence for any para-hypercomplex basis $(I,J,K)$ of $\widetilde{Q}$ one has
$IU \subset U, \; JU \subset U, \; KU \subset U$.
\end{defi}

The sum and the intersection of  para-quaternionic subspaces is  para-quaternionic.

\begin{prop} \label{para-quaternionic subspaces proposition}
Let $(E')^{k} \subset E$ be any subspace. Then
\begin{equation} \label{para-quaternionic subspace}
 U^{2k} =
H \otimes E'
\end{equation}
is a para-quaternionic subspace of dimension $2k$.
Viceversa any para-quaternionic subspace   of $V$ has this form.
Moreover  $U$ is para-quaternionic Hermitian (with neutral metric)  iff $E'$ is $\omega^E$-symplectic.
\end{prop}

\begin{proof}
The subspace $U^{2k}=H^2 \otimes (E')^{k} \subset H
\otimes E$ is clearly $\widetilde{Q}$-invariant. Viceversa, let $U \subset V$ be a para-quaternionic subspace.
Let fix a basis $(h_1,h_2)$ in $H$  and
let moreover $(I,J,K)$ be the associated standard para-hypercomplex Hermitian structure. The subspaces $E' =
p_1(U)$ and $p_2(U)$ coincide since, by the $I$-invariance, for any $X=h_1
\otimes e + h_2 \otimes e' \in U$, the vector $IX=h_2 \otimes e -
h_1 \otimes e' $ is in $U$. Also, since $KX= -h_1 \otimes e + h_2 \otimes e'$, by the
$K$-invariance  the decomposable vectors $h_1
\otimes e$ and $ h_2 \otimes e'$ are in $U$. Therefore $U=H
\otimes E'$.



From the metrical  point of view,
the subspaces $h_1 \otimes E'$ and $h_2 \otimes E'$ are totally
isotropic and the metric on $U$, with respect to the  decomposition  $U= h_1 \otimes E' \oplus h_2 \otimes E'$, is given by
\begin{equation} \label{metric in para-quaternionic
subspace}
 g|_{U}=\left(
   \begin{array}{c c}
0 & \omega^E|_{E'}\\
 (\omega^E|_{E'})^t & 0
\end{array}
\right). \end{equation}
Then $U$ is Hermitian para-quaternionic if and
only if $E'$ is $\omega^E|_{E'}$-symplectic.
\end{proof}

\begin{remk}
Referring to the decompositions $1),\; 2)$ given in prop. (\ref{decomposition of a subspace}),
notice that a para-quaternionic subspace  $U=H \otimes E'$ decomposes respectively as
\[
1) \quad U= h_1 \otimes E' \oplus h_2 \otimes E', \quad \qquad 2) \quad U= h_1 \oplus E' \oplus \{h_1 \otimes e' + h_2 \otimes Te', \; e' \in E'\}
\]
w.r.t. any basis $(h_1,h_2)$, with $T$  any automorphism of $E'$ with no real eigenvalues.
In this case then, the  dimensions of the maximal $U^{F,T}$ subspaces in $U=H \otimes E'$
 in the decompositions 1) and  2) coincide and equal the dimension of $E'$.
\end{remk}


Any subspace $U$ of $V$ contains a (possibly zero) {\it maximal para-quaternionic subspace}
$U_0= U \underset{A \in \widetilde Q} {\cap} A(U)$.
Equivalently, $U_0= U \cap IU \cap JU\cap KU$ for any admissible basis $(I,J,K)$ of $\widetilde Q$.

\begin{defi} A subspace $U \subset V$ is called \textbf{pure} if $U_0=\{0\}$, i.e. it does not contain any non zero para-quaternionic subspace.
\end{defi}
Clearly, from prop. (\ref{spazi U(F,L)}a ), \textit{any $U^{F,T}$ subspace is pure}.



\subsection{Complex subspaces} \label{section:complex subspaces}

\begin{defi}
A  subspace $U\subset V$ is called \textbf{complex} if there
exists a compatible complex structure $I \in \widetilde{Q}$ such
that $U$ is $I$-invariant i.e $IU \subset U$. We denote it by  $(U,I)$.
\end{defi}
We shall include $I$ into an admissible basis $(I,J,K)$ of $\widetilde{Q}$.
 Such basis  will be called \textbf{adapted} to the subspace $(U,I)$. Adapted bases are defined up to a rotation in the real plane spanned by $J$ and $K$.

\begin{lemma} The complex structure $I$ is unique up to sign unless $U$ is para-quaternionic.
\end{lemma}
\begin{proof}
Let $\tilde{I}= aI + bJ + cK, \quad  \tilde{I}^2=-Id \quad (a^2-b^2-c^2=1)$, be a compatible complex structure such that $\tilde{I}U=U$.
Then for any $X \in U$ one has $aIX + bJX - cJ I X \in U$, hence
$ J(bX-cIX) \in U , \quad \forall X \in U$. If $(b,c)=(0,0)$ then $\tilde{I}= \pm I$; otherwise the map $X \mapsto (bX-cIX ), \; \forall X \in U$
 is an automorphism of $U$, since $I$ has no real eigenvalues, hence $JU=U$ i.e. $U$ is para-quaternionic.
\end{proof}

\begin{lemma}\label{equivalent definition of pure-complex subspace}
A complex subspace $(U,I) \subset (V,\widetilde{Q})$  is pure if and only if there exists
a para-complex structure $J
\in \widetilde{Q}$ such that $ IJ=-J I, \; JU  \cap U=\{0\}$.
\end{lemma}

\begin{proof}
Let $(U,I)$ be pure complex. Suppose there exists   $X \in U$ such that
 $JX \in U$  with $J$
a compatible para-complex structure and $IJ=-JI=K$.
Then  $KX=IJX \in U$ and $<X,IX,JX,KX>_\R \subset U$ is  a
para-quaternionic subspace. Hence
$X\neq 0$ would give a contradiction. Viceversa is obvious.
\end{proof}
It is also immediate to verify that
if $(U,I)$ is pure complex, then, for \underline{any} $A \in  \widetilde Q, A \neq \pm aI, \; a \in \R$, one has $A U  \cap U=
\{0\}$.


Considering now also the metric structure of $V$ we have the following special class of pure complex subspaces.
\begin{defi}  \label{definition of totally complex subspace}
An Hermitian  complex subspace $(U,I)$  of $V$
 is called \textbf{totally complex}  if
there exists an adapted hypercomplex basis $(I,J,K)$ such that $
JU\perp U$ ($\Leftrightarrow KU \perp U$) with respect to the
(non degenerate) induced  metric $g$.
\end{defi}

Note that the Hermitian complex subspace $(U,I)$ is totally complex iff, with
respect to the adapted basis $(I,J,K)$,  the restriction to $U$ of the Hermitian product
(\ref{Hermitian scalar product on a PQH vector space}) has complex values\footnote{In fact,
following terminology of \cite{B},  a  totally complex  subspace could be called a
{\it subspace  with complex Hermitian product.}}. Note also that the hypothesis  of Hermitianicy
is necessary to ensure that  any  totally complex subspace is pure.

Let  $(U,I)$ be a totally complex subspace  and   $(I,J,K)$  an adapted basis such that
$JU \perp U$. Any   $A= aI + bJ + cK \in \widetilde Q$,   satisfies $A U \perp U$ if and only if $a=0$.
Then again, adapted bases are defined up to a rotation in the plane $<J,K>$.



By taking into account that $U_0$ is $I$-complex  and from known facts about complex structures,
one has the following
\begin{prop} \label{unicity of complex structure}
Any complex subspace $(U,I)$ is a direct sum of the maximal para-quaternionic subspace
$U_0$ and a pure $I$-complex subspace $(U',I)$ i.e. $U= U_{0} \oplus U'$.
If  $(U'',I)$ is another $I$-pure complex subspace complementary to $U_0$, then $U'$ and $U''$ are isomorphic
as $I$-complex spaces.
\end{prop}

Hence \textit{the description of complex subspaces reduces to the description of pure complex
subspaces}.

Let $I$ be a compatible complex structure and
$(h_1,h_2)
$ a symplectic basis of $H$ such that  $I \equiv
\mathcal{I} \otimes Id$ using the notations introduced for
proposition (\ref{isomorphism between para-quaternionic space
 and tensor product H x E and sl(H)}), i.e. $I$ as in
(\ref{standard hypercomplex structure}).

\begin{teor}
With respect to $(h_1,h_2)$, a subspace $U \subseteq V$ is $I$-pure complex iff  $U=U^{F,T}$ with $T$ a complex structure on $F=p_1(U)$.
Then the map
\[
\begin{array}{llll}
\phi: &  (F,T^{-1}) & \rightarrow & (U^{F,T},I)\\
& f & \mapsto & h_1 \otimes f + h_2 \otimes Tf
\end{array}
\]
is an isomorphism of complex vector spaces.
The signature of the metric on $U$ is of type
$(2p,2s,2q), \, 2s=\dim \ker g|_U$, and $U$ is  Hermitian if and only if  $F$ is
$g_F$-non degenerate.
In this case
$\phi: (F,T^{-1},g_F) \rightarrow (U,I, g|_U)$ is an isomorphism
of Hermitian spaces. In particular $T^{-1}$ is $g_F$-skew
symmetric. The Kaehler form of $(U,I)$ is given by
\[
\phi^*(g|_U \circ I)=g_F \circ T^{-1}= -(\omega^E|_F + \omega^E|_F(T \, \cdot,T \, \cdot)).
\]

The subspace $(U,I)$ is totally complex if and only $F$ is $\omega^E$-symplectic and  $T$ is
$\omega^E|_F$-skew-symmetric ($\Longleftrightarrow$ if $T$
 preserves the form $\omega^E|_F$
i.e.
\begin{equation}\label{T preserves omega}
\omega^E|_F(f,f')= \omega^E|_F(Tf,Tf')\qquad\qquad \forall f,f' \in F)
\end{equation}
 or equivalently $g_F= -2 \omega^E|_F \circ T \quad $.
\end{teor}

\begin{proof}
Let $(U,I)$ be a pure complex subspace in $H \otimes E$, $(h_1,h_2= \mathcal{I} h_1)$ a symplectic basis of $H$ s.t. $I \equiv
\mathcal{I} \otimes Id$ and $(I,J,K)$ an adapted basis.
Observe that there is no
non-zero decomposable element $X=h \otimes e$ in $U$; in fact
since $IX \in U$,
it would follows that $h_1 \otimes e$ and $h_2 \otimes e$ are both in $U$, hence $H \otimes \R e \subset E$ which is a contradiction.
From lemma (\ref{subspaces with no decomposable}),  $U=U^{F,T}$  w.r.t. the symplectic basis $(h_1,h_2=\mathcal{I} h_1)$; then
 it is  straightforward to verify that  $T^2=-Id$.
Viceversa, it is immediate to verify that the pure  subspace $U= U^{F,T}= \{h_1 \otimes f + h_2 \otimes Tf, \; T^2= -Id\}$ is  $I$-invariant.
The statements about the isomorphism $\phi$ are straightforward to verify.
The expression of the Kaehler form follows from a direct computation.


  Any pure  $I$-complex subspace admits a decomposition into pure
 $I$-complex 2-planes; each one of them is whether totally isotropic or with definite metric.
This implies that the signature of the metric on $U$ is $(2p,2s,2q), \; 2s= \dim \ker g|_U$ and clearly equals the
signature of  $g_F$ on $F$.  Consequently $U$ is Hermitian pure complex if and only
if  $F$ is $g_F$-non degenerate.

We now prove the last statement. Observe that, since w.r.t. $(h_1,h_2)$ $T$ is a complex structure on $F$, the claimed equivalence $\Longleftrightarrow$ is straightforward. For any $X= h_1 \otimes f + h_2 \otimes Tf, \; \in U$, we have
$JX=h_2 \otimes f + h_1 \otimes Tf$, hence
$ JU=\{Y=h_1 \otimes f - h_2 \otimes Tf, \quad f \in F\} $.
(Observe that for any pure complex subspace $(U,I)$, the subspace $JU$ is pure $I$-complex).
Then
 $U \perp JU$ if and only if, for any $f,f' \in F$,
\[
0=g(h_1 \otimes f + h_2 \otimes Tf, h_1 \otimes f' - h_2 \otimes
Tf')= -\omega^E(f,Tf')- \omega^E(Tf,f')
\] that is $\omega^E(f,Tf')= -\omega^E(Tf,f')$ which is equivalent to (\ref{T preserves omega}). The metric on $U$
verifies
\[g(h_1 \otimes f + h_2 \otimes Tf, h_1 \otimes f' + h_2 \otimes
Tf')= g_F(f,f')= 2 \omega^E(f,Tf').\] Therefore,  the non degeneracy of $U$ implies
that  $F$ is $\omega^E$-symplectic.
\end{proof}


\textit{This theorem reduces classification of pure complex subspaces to the classification of pairs $(F,T)$ with $F
\subset E$ and $T$ a complex structure on $F$}. In particular, in the  classification of totally complex
subspaces,  $F$ is, in addition,  $\omega^E|_F$-symplectic and $T$ preserves $\omega^E|_F$.

In the following section we  consider the para-complex subspaces. They  could be partly treated in a unified way together with the complex subspaces just seen.
But the existence of specific characteristics not appearing in the complex case, account for a separate
treatment.



\subsection{Para-complex subspaces} \label{section:para-complex subspaces}


\begin{defi}
A  subspace $U\subset V$ is called
\textbf{weakly para-complex} if there exists a para-complex
structure $K \in \widetilde{Q}$ such that $U$ is $K$-invariant
i.e $KU \subset U$. We denote such subspaces by $(U,K)$.
A \textbf{para-complex} subspace $(U,K)$ is a weakly para-complex subspace
such that  $ \dim \, U_{K}^+ =  \dim \, U_{K}^-$.
\end{defi}

\begin{remk} \label{eigenspaces of a paracomplex structure are decomposable subspaces}
  The eigenspaces $V^+_K,V^-_K$  of a given para-complex structure $K \in \widetilde Q$ are decomposable subspaces (then totally isotropic) of $V$, i.e. $V^+_K=h'\otimes E', V^-_K=h''\otimes E''$ and $E'\oplus E''=E$. As a first consequence (cfr. footnote in definition (\ref{footnote about metric on product structure subspaces})), any weakly para-complex subspace not para-complex is degenerate.

\end{remk}
  The presence of decomposable vectors  produces a difference passing from the complex to the weakly para-complex case but,
  as we will see, a common treatment of both cases is still possible.

Analogously to lemma (\ref{equivalent definition of pure-complex subspace}) one has
\begin{lemma} \label{equivalent definition of para-complex subspace}
 A  weakly para-complex subspace $(U,K)$ of $V$ is
 pure  if and only if there exists a complex
structure $I \in \widetilde{Q}$  anti-commuting with $K$ such that $ IU \cap U=\{0\}$.
\end{lemma}

Remark that \textit{if $(U,K)$ is pure weakly para-complex, then for any compatible complex structure $\tilde I$, one has $\tilde IU \cap
U=\{ 0\}$}. In fact, let   $(I,J,K)$ be  the adapted basis of the para-complex subspace $(U,K)$ with
$IU \cap U=\{0\}$. Let $\widetilde Q \ni \tilde I= aI + bJ + cK, \; a^2-b^2 -c^2=1 $ be a compatible complex
structure. Suppose there exists a non zero $X \in U$ such that $\tilde I X\in U$; then  $I(aX - bKX) \in U$.
This implies $a=\pm b$, hence a contradiction. Then the admissible bases are defined up to a
pseudo-rotation in the plane
$<I,J>_\R$.


\begin{prop}\label{decomposition of a para-complex subspace}
Any weakly para-complex subspace $(U,K)$ is a direct sum $U= U_{0} \oplus \tilde U$ of the maximal para-quaternionic subspace
 and of a pure weakly $K$ para-complex subspace $\tilde U$.
If  $\tilde U'$ is another $K$ pure weakly para-complex complementary subspace, then $\tilde U$ and $\tilde U'$ are isomorphic as  weakly $K$-para-complex spaces.

Assume  $\tilde U \neq \{0\}$. If
$\tilde U \nsubseteq U_K^{\pm}$ then  the para-complex structure $K \in
\widetilde{Q}$ is unique up to sign.
Otherwise the  family of para-complex structures
\[
 \tilde K_a=aI + a J  \pm K, \quad  \text{if}  \quad \tilde U \subset {U_{K}}^+ \; ; \qquad  (\tilde K_a=aI - a J  \pm K, \quad  \text{if} \quad \tilde U \subset {U_{K}}^-)
\]
preserves $U$ for any adapted basis $(I,J,K)$.
\end{prop}

\begin{proof}
The proof of the first statement is analogous to that one in the proof of proposition (\ref{unicity of complex structure}).
Let $(U,K)$ be a weakly para-complex subspace, $(I,J,K)$ an adapted
basis and $\tilde K= aI+ bJ+ cK \in \widetilde Q, \quad \tilde K^2=Id,
\quad (a^2 -b^2-c^2=-1)$ an admissible para-complex structure.
For any $X \in U$, the vector $\tilde K X$ is in $U$  iff
\begin{equation} \label{unicity of para-complex structure}
aIX+ bJX= I(aX -bKX) \in U.
\end{equation}
Then, whether $a=b=0$ i.e. $\tilde K= \pm K$, or $U$ is para-quaternionic for $a \neq \pm b$ (in fact in this case the map $X \mapsto aX -bKX$ is an automorphism of $U$).
In case  $a=\pm b$,
let  consider the decomposition $U=U_0 \oplus U'$ into the maximal para-quaternionic subspace $U_0$  and the weakly pure para-complex component
 $U'$ respectively. Then condition (\ref{unicity of para-complex structure})
implies that $aX -bKX=0, \; \forall X \in U'$
which is verified only if $U'$ is an eigenspace of $K$.
\end{proof}

Hence also \textit{the  description of weakly para-complex subspaces reduces to that one  of pure  subspaces}.
In this case nevertheless there exists a difference regarding the unicity of the para-complex structure.
The reason for such a difference is a consequence of the results in  the next subsection
(see in particular the end of the proof of prop. \ref{proposition about decomposition of a nilpotent subspace}).


\begin{defi} Let $(U,K)$ be a  $K$-Hermitian    para-complex subspace.
Then $U$ is called \textbf{totally para-complex} if there
exists a complex structure $I \in \widetilde Q$ anticommuting with
$K$ such that $IU \perp U$ respect to the induced metric $g|_U$.
\end{defi}
Observations analogue to those following the definition (\ref{definition of totally complex subspace})
of totally complex subspaces can be made for totally para-complex subspaces.





Let $J$ be a compatible para-complex structure and $(h_1,h_2)$  a symplectic basis of $H$ such that
$J= \mathcal{J} \otimes Id$ i.e. $J$ as in  (\ref{standard hypercomplex structure}).
\begin{teor}
With respect to $(h_1,h_2)$, a subspace $U \subseteq V$ is pure weakly $J$-para-complex iff
 $U=U^{F,T}$ with $T$ a weakly para-complex structure on $F=p_1(U)$.
Then the map
\[
\begin{array}{llll}
\phi: &  (F,T) & \rightarrow & (U^{F,T},J)\\
& f & \mapsto & h_1 \otimes f + h_2 \otimes Tf
\end{array}
\]
is an isomorphism of weakly para-complex vector spaces.
The subspace $U$ is  $J$-Hermitian  if and only if  $F$ is
$g_F$-non degenerate hence necessarily para-complex.
In this case, the signature of $g|_U$ is always neutral and
   $\phi: (F,T,g_F) \rightarrow (U,J, g|_U)$ is an isomorphism of Hermitian  para-complex spaces. In particular $T$ is $g_F$-skew symmetric.

The para-Kaehler form is given by $\phi^*(g|_U \circ J)= g_F \circ T= -(\omega|_F - \omega|_F(T \, \cdot,T \, \cdot))$.

The  para-complex subspace $(U,J)$ is  totally para-complex if and only if $T$ is $\omega^E|_F$-skew-symmetric $\Longleftrightarrow$ the form $\omega^E|_F$ is skew-invariant w.r.t. $T$
 i.e.
 \[
 \omega^E|_F(f,f')= -\omega^E|_F(Tf,Tf') \qquad \qquad \forall f,f' \in F
 \]
  or equivalently $ g_F= -2 \omega^E|_F \circ T$, and $F$ is $\omega^E|_F$-symplectic.
\end{teor}

\begin{proof}

Let $(U^{k},J)$ be a pure weakly para-complex subspace in $H \otimes E$, $(h_1,\mathcal{J} h_1=h_2)$ a symplectic basis and $(I,J,K)$ an adapted basis.
Clearly $h_i \otimes E \cap U= \{0 \}, \; i=1,2$ since $U$ is pure. Then, from proposition (\ref{spazi U(F,L)}),
$U=U^{F,T}$. In particular, w.r.t. a symplectic basis $(h_1,\mathcal{J} h_1=h_2)$,
It is straightforward to verify that $T^2= Id$.
Viceversa the pure subspace $U=U^{F,T}=\{h_1 \otimes f+ h_2 \otimes Tf, \; T^2=Id \}$ is clearly $J$-invariant.

The eigenspaces $V_J^+$ and $V_J^+$   are decomposable subspaces
(see remark (\ref{eigenspaces of a paracomplex structure are
decomposable subspaces}) and consequently they are totally
isotropic.  Then the signature of the induced metric on $U$ (which
equals the signature of $g_F$ on $F$) is
$(m,k-2m, m)$ where $m= rk \, g(V_J^+ \times V_J^-)$. From  (\ref{Hermitian metric on F}), $U$ if $J$-Hermitian with neutral signature iff $F$ is $g_F$ non degenerate.


Consider now a $J$-Hermitian para-complex subspace $U=\{h_1 \otimes f + h_2 \otimes Tf, f \in F \}$. Then $IU=\{h_1 \otimes -Tf + h_2 \otimes f, f \in F \}$.
Imposing $U \perp IU$ it follows that the condition for $(U,J)$ to be  totally para-complex is given by
\begin{equation} \label{totally para-complex condition}
\omega|_F(f,f')= -\omega|_F(Tf,Tf')
\end{equation}
($\Leftrightarrow T$ is $\omega^E|_F$-skew-symmetric). The Hermitianicy hypothesis on $F$ implies that $F$ if  $\omega^E$-symplectic.
Then the decomposition  $F=E_1 \oplus E_2$ into into $\pm 1$-eigenspaces of $T$ is a Lagrangian decomposition
(i.e. $\omega^E|_{E_1} \equiv 0, \quad \omega^E|_{E_2} \equiv 0$) of the symplectic space $F$.
\end{proof}



The last  theorem reduces classification of weakly pure complex
subspaces to that one of pair $(F,T)$ with $F \subseteq E$
and $T$  a weakly para-complex structure on $F$.



Differently from the pure complex case, where $(U,I)$ admits the
form $U^{F,T}$ w.r.t. all symplectic bases of $H$, in the pure
weakly para-complex case, the presence of decomposable vectors in
$(U,J)$ and lemma (\ref{subspaces with no decomposable}) allow
for some special presentations of $(U,J)$ different from the
$U^{F,T}$ form. In particular, using the decomposition of $(U,J)$
into the $\pm 1$ eigenspaces of $J$ on $U$, we have the following
\begin{prop}
Let $(U,J)$ be a pure weakly para-complex subspace with $(h_1,h_2=\mathcal{J}  h_1)$ a symplectic basis. Let moreover $(I,J,K)$ be an adapted basis.
The pure weakly para-complex subspace decomposes as
\[
(U,J)= (h_1' \otimes E_1) \oplus (h_2' \otimes E_2)
\]
where $E_1 \oplus E_2=F$ is the $T$ $\pm 1$-eigenspaces decomposition of $F$,  $h_1'= -\frac{1}{\sqrt{2}}(h_1 + h_2),\;  h_2'=  \frac{1}{\sqrt{2}}(h_1 - h_2)$
the symplectic basis of eigenvectors of $\mathcal J$ and  $h_1' \otimes E_1= U_J^+$ and $h_2' \otimes E_2=U_J^-$ are the eigenspaces of $J|_U$.
\end{prop}




\subsection{Nilpotent subspaces} \label{section:Nilpotent subspaces}
\begin{defi} A  subspace $U \neq \{0 \} \subset H \otimes E$ is called
\textbf{nilpotent} if there exists a non zero nilpotent endomorphism $A \in \widetilde Q$ which preserves $U$.
\end{defi}
The nilpotent subspace  $U$  will be called also $A$-nilpotent even if, as
we will see  later, such a  nilpotent endomorphism  is never unique.

If $U$ is nilpotent we call  \textbf{degree of nilpotency of $U$}
the minimum integer $n$ such that $A^n U=\{ 0\}, \; A \in \widetilde Q$.
Clearly, since $A^2=0$, the degree of nilpotency  of $U$ is at most 2, and equal to 1 if $U \subset \ker A$.

\begin{prop} \label{nilpotent subspaces}
A subspace $U$ is nilpotent of degree 1 iff  it is a decomposable  subspace $h \otimes F, \; F \subset E$. More generally,
let $A \in \widetilde Q$ be a nilpotent endomorphism and $\ker A= h\otimes E$.
A subspace $U$ is $A$-nilpotent  iff, with respect to a symplectic basis $(h_1\equiv h,h_2)$, one has
\[h_1 \otimes p_2(U) \subset U.\]
\end{prop}

\begin{proof}
We first observe that the subspace $p_2(U)$ is  invariant for any change of symplectic basis $(h_1, h_2) \mapsto (h_1, h_2'), \;$.
 Fixed a basis $(h_1,h_2)$ in $H$, let $(I,J,K)$ be the associated standard para-hyper\-com\-plex structure.
Let $U$ be a $A$-nilpotent subspace of degree 1 with $A= \alpha I + \beta J + \gamma K$ and
$\| A \|^2 \equiv \alpha^2-\beta^2-\gamma^2=0$. Then, for any $X= h_1 \otimes e_1 + h_2 \otimes e_2 \in U$,
condition
$AX= 0$ implies
$e_2 =
\frac{\gamma }{(-\alpha + \beta)}e_1$, i.e.  $U$ is the decomposable subspace
$(h_1 +  \frac{\gamma }{(-\alpha + \beta)}h_2) \otimes p_1(U)$. Viceversa it is clear that any decomposable
subspace $U= h \otimes F$ is nilpotent of degree 1; moreover, all $A \in
\widetilde Q$  with $\ker A= h \otimes E$ annihilate $U$.

More generally, let $U$ be a $A$-nilpotent subspace where  $A \in \widetilde Q$ with $\ker A= h\otimes E$.
Let $(h_1\equiv h, h_2)$ be a symplectic basis of $H$ (then $A(h_2 \otimes E)= h_1 \otimes E$).
 For any $X=h_1 \otimes e_1 + h_2 \otimes e_2 $ with $ e_2 \neq 0$ in $U$ the vector $AX \in U$ implies that
 $h_1 \otimes e_2 \in U$. So, being $E_1= p_1(U), \; E_2= p_2(U)$, the $A$-invariance of $U$ implies that
  $h_1 \otimes E_2 \subset U$ ($\Rightarrow$ $E_2 \subseteq E_1$). Viceversa,
  let $U$ be a subspace. If w.r.t. a symplectic basis $(h_1,h_2)$ the subspace  $h_1 \otimes p_2(U) \subseteq U$, then, for
 any $A \in \widetilde{Q}, \; \| A \|^2=0$ with  $\ker A= h_1 \otimes E$,
 the subspace $U$ is clearly $A$-nilpotent.

\end{proof}

Clearly all para-quaternionic subspaces are nilpotent of degree 2 w.r.t. any nilpotent structure in
$\widetilde Q$.

From previous proposition, we have the  following  characterization of  nilpotent subspaces with respect
 to proposition (\ref{decomposition of a subspace}).
\begin{prop} \label{proposition about decomposition of a nilpotent subspace}
Let $A \in \widetilde Q$ be a nilpotent endomorphism such that $\ker A= h_1 \otimes E$ where $(h_1, h_2)$ is a
symplectic basis. The subspace
\begin{equation} \label{decomposition of a nilpotent subspace}
U= (H \otimes E_0) \oplus (h_1 \otimes E_1'') \oplus \{ h_1 \otimes \bar{e_1} + h_2 \otimes  T' \bar{e_1}, \; \bar{e_1} \in \bar{E_1} \}
\end{equation}
with $\bar{E_1}\cap E_1''= \{0\}$, $T' \bar{E_1} \subset E''$ and $T'$ injective is $A$-nilpotent. The subspace \[U'=(h_1 \otimes E_1'') \oplus \{ h_1 \otimes \bar{e_1} + h_2 \otimes  T' \bar{e_1}, \; \bar{e_1} \in \bar{E_1} \}\]
 is pure nilpotent of the form $U'=U^{F,T}$ with
 $F= E_1'' \oplus \bar{E_1}$ and $T= 0 \oplus T': E_1'' \oplus \bar{E_1} \rightarrow E_1''$.

 Viceversa, any $A$ nilpotent subspace can be written in the form  (\ref{decomposition of a nilpotent subspace})
i.e. it is direct sum of a para-quaternionic subspace, a decomposable subspace $(h_1 \otimes E_1'')$ and a subspace
 $\{ h_1 \otimes \bar{e_1} + h_2 \otimes  T' \bar{e_1}, \; \bar{e_1} \in \bar{E_1} \}$ with $T'$ injective, $\bar{E_1}\cap T\bar{E_1}= \{0\}$ (in next section such a subspace will be called real) and $T' \bar{E_1} \subset E_1''$.

Moreover a sufficient  condition for $U$ to be not degenerate is that $p_2(U)$ is $\omega^E$-symplectic.

\end{prop}
\begin{proof}
The subspace $U$ in (\ref{decomposition of a nilpotent subspace}) is clearly $A$-nilpotent w.r.t. all $A \in \widetilde Q$ such that $\ker A= h_1 \otimes E$.
Viceversa if $U$ is a $A$-nilpotent subspace with $\ker A= h_1 \otimes E$, from proposition (\ref{nilpotent subspaces})
we have $h_1 \otimes p_2(U) \subset U$ w.r.t. all symplectic basis $(h_1,h)$.
Let then fix a basis $(h_1,h_2)$. Let $(h_1 \otimes E) \cap U= h_1 \otimes E_1'$ and $p_1(U)= E_1'  \oplus \bar{E_1}$.
Then
\begin{equation} \label{nilpotent subspace expession}
\begin{array} {lll}
U  =  (h_1 \otimes E_1') \oplus \{ h_2 \otimes e_2 + h_1 \otimes \tilde T e_2, \; e_2 \in E_2\}, \;   E_2 \subseteq E_1', \; E_2 \cap  \tilde T E_2= \{ 0\},
\end{array}
\end{equation}
where $\tilde T E_2=  \bar {E_1}$ and  the complement
\begin{equation} \label{nilpotent subspace U_FT part}
\tilde U= \{ h_2 \otimes e_2 + h_1 \otimes \tilde T e_2\}
\end{equation}
is of type $U^{F,T}$ with $\tilde T: E_2 \rightarrow \bar{E_1}$.

We know that a necessary condition for a subspace $U$ to be nilpotent is the presence of a decomposable subspace in $U$.
 More precisely, condition $E_2 \subseteq E_1'$ in (\ref{nilpotent subspace expession}), implies that whether $U$
 contains a para-quaternionic subspace (in case $\tilde T$ is not injective) and then $h \otimes E \cap U \neq 0, \; \forall h \in H$,
  or $U=U^{F,T}$ (if $\tilde T$ is injective) and in this case  one and only one decomposable subspace is in the pure nilpotent subspace $U$.
(In the next section we will see that, in this second case, the addend $\tilde U$ in (\ref{nilpotent subspace U_FT part}) is a real subspace).

In  case $U$ is not pure, let $E_2= E_0 \oplus E_2'$ and  $E_1'= E_0 \oplus E_1''$ with  $E_0= \ker \tilde T$ and $E_1'', \, E_2'$ some  complementaries.
 Then
\begin{equation} \label{nilpotent decomposition}
U= (H \otimes E_0) \oplus (h_1 \otimes E_1'') \oplus \{ h_2 \otimes e_2' + h_1 \otimes  (T')^{-1} e_2',  \quad e_2' \in E_2' \}
\end{equation}
with $(T')^{-1}: E_2' \rightarrow \bar{E_1}$ an isomorphism and $E_2' \subset E_1''$.


Let us look for the sufficient condition $U$ to be non degenerate.
Let $X_0 \in U$ and suppose $g(X_0,Y)=0, \; \forall Y \in U$ with $X_0=h_1 \otimes e_0' + h_2 \otimes e_0 + h_1 \otimes \bar e_0$  and
$Y=h_1 \otimes f' + h_2 \otimes f + h_1 \otimes \bar f, \; e_0', f' \in E_1', e_0, f \in E_2, \bar{e_0}, \bar{f} \in \bar{E_1}$ according to the decomposition
given in (\ref{nilpotent subspace expession}). Then
\[
-\omega^E(e_0, f'+ Tf) + \omega^E(e_0' -T e_0,f)= 0, \; \forall f \in E_2, \forall f' \in E_1'.
\]
This implies that $\omega^E(E_2,E_1')$ is degenerate. Then conclusion follows.

\end{proof}
Note that, from (\ref{nilpotent decomposition}), every non trivial pure nilpotent subspace contains
a non trivial pure weakly para-complex subspace.
Moreover any  pure weakly para-complex subspace is direct sum of a  pure para-complex subspace and a degree 1 nilpotent subspace.



\subsection{Real subspaces} \label{section: Totally real subspaces}
\begin{defi} A  subspace $U \subset V$ is called
\textbf{real} if $AU \cap U= \{0\}, \quad \forall A \in \widetilde Q$.
Equivalently, $U$ does not contain either a non trivial  complex or weakly  para-complex
subspace.
\end{defi}

Let prove the above equivalence. If $AU \cap U = \{ 0 \}, \; \forall A \in \widetilde Q$,
clearly no non trivial complex or weakly para-complex subspaces are in $U$.
Viceversa let $U$ contain no non trivial complex or weakly para-complex  subspaces.
Then, as remarked in the previous section, it contains no non trivial nilpotent subspaces as well.

 \textit{A real subspace $U$ is pure}. Also,  $ \dim U \leq \frac{1}{2}\dim V$.

\begin{defi} A  non degenerate real subspace $U \subset V$ is
called  \textbf{totally real} if for one and hence for any para-hypercomplex basis $(I,J,K)$ of $\widetilde Q$,
\[IU \perp U, \qquad JU \perp U, \quad KU \perp U\]
or equivalently if the Hermitian product (\ref{Hermitian scalar
product on a PQH vector space})  has real values   for any admissible basis
$(I,J,K)$ of $\widetilde Q$ \footnote{In fact, in \cite{B} such
subspace is called a {\it subspace with real Hermitian product.}}.
\end{defi}
The implication in the first statement is straightforward to verify.
In this case $ \dim U \leq \frac{1}{4}\dim V$.

\begin{teor} A subspace $U \subseteq V$ is real  iff w.r.t. a symplectic basis $(h_1,h_2)$ it is $U=U^{F,T}$ where
 the linear map  $T: F=E_1=p_1(U) \rightarrow
p_2(U)$  is an isomorphism such that, for any non trivial subspace
$W \subset F \cap TF$, it is  $TW \nsubseteq W$.

The subspace $U$ is non degenerate if and only $F$ is $g_F$-non degenerate.

Let $E_2=TE_1$. The real subspace $U$ is totally real if and only if
\begin{equation} \label{totally real conditions}
 \omega^E|_{E_1}=\omega^E|_{E_2} \equiv 0 \qquad  \text{and} \quad T|_{E_1} \; \; \text{ is} \; \omega^E|_F-\text{skew-symmetric}.
\end{equation} which implies $E_1 \cap E_2= \{0\}$.
\end{teor}

\begin{proof}
Let $U$ be a real subspace.
Since no non trivial weakly para-complex subspace is in $U$ that it  contains no decomposable
 vectors and from lemma (\ref{subspaces with no decomposable}), fixed any  symplectic basis  $(h_1,h_2)$ of $H$, we can write
\[U=U^{F,T}=\{h_1 \otimes e + h_2 \otimes Te, \; e \in F=E_1= p_1(U) \}.\]
Suppose   $T \tilde F \subset  \tilde F$ for some subspace $\tilde
F \subset W=E_1 \cap E_2$,
 Then $\tilde F$ must be an even dimensional subspace direct sum of 2-dimensional
 $T$-invariant real subspaces $\tilde F_i$. We show that necessarily $\tilde F= \{0 \}$.

Let then   $\tilde F\supseteq \tilde F_i= <e,Te>_\R$ be a $T$-invariant plane, with $T(Te)= \lambda e + \mu Te$.\\
 Observe that both $\mu$ and $\lambda$ can not be zero. In fact,
 if $\lambda=0$ then $T(Te)=\mu Te$ which is excluded since $T|_W$ has no invariant lines.
 If $\mu=0$ then $T(Te)=\lambda e$ with $\lambda \leq 0$ since the vectors $e,Te$ are linearly independent.
 Then the map  $\tilde T=\frac{T}{\sqrt{|\lambda|}}$ is a complex structure on $<e,Te>$  and the
 subspace $\tilde U= \{h_1 \otimes f + h_2 \otimes  \tilde T  f, \; f \in <e,Te>\}$ is a  complex subspace  in $U$.
 So necessarily $\mu \neq 0$.

 Consider the  non null vector $X=h_1 \otimes e + h_2 \otimes Te \in U$.   For any $A \in \widetilde Q, \; A= \alpha I + \beta J + \gamma K$ with $I,J,K$
 the para-hypercomplex basis
associated to the chosen basis $(h_1,h_2)$, by hypothesis, whether $AX=0$ or $AX \notin U, \; \forall \alpha, \beta, \gamma$. Computing
 The vector $AX=0$ only if $A$ is the null map. But, for any $\gamma$ and  by choosing \[\alpha= \frac{\gamma}{\mu}(\lambda-1), \qquad \beta= \frac{\gamma}{\mu}(1+\lambda)\] the vector $AX \in U$ since,  in this case, $Te'= \tilde e$, contradiction.

Viceversa,  let  $U=U^{F,T}$ w.r.t. the symplectic basis $(h_1,h_2)$; denote $E_1=F, \; E_2=TF$,
and assume that $T: E_1 \rightarrow T(E_1)=E_2$ is an isomorphism
such that for any non trivial subspace $W \subset E_1 \cap E_2$, it
is $TW \subsetneq W$. Let  $A= \alpha I+ \beta J + \gamma K, \,
\in \widetilde Q$. Suppose there exists a non null vector  $X= h_1
\otimes e + h_2 \otimes Te \in U$, such that
 $AX= h_1 \otimes (-\gamma e + (\beta-\alpha)Te) + h_2 \otimes ((\alpha+\beta)e + \gamma Te) \neq 0$
 is in $U$. This implies that $T^2 e \in <e,Te> \subset (E_1 \cap E_2)$, which gives a contradiction.

From (\ref{Hermitian metric on F}), the subspace $U$ is non degenerate if and only $F$ is $g_F$-non degenerate.

Let $ U= \{X=h_1 \otimes e + h_2 \otimes Te, \quad e \in E_1 \}$
be a totally real subspace  in $V$. Then
\begin{equation} \label{subspaces IU,JU,KU in totally real }
\begin{array}{l}
1) \quad IU=\{Y= -h_1 \otimes Te_1 + h_2 \otimes e_1,  \quad e_1 \in E_1 \},\\
2) \quad JU=\{Y= h_1 \otimes Te_2 + h_2 \otimes e_2,  \quad e_2 \in E_1 \},\\
3) \quad KU=\{Y= -h_1 \otimes e_3 + h_2 \otimes Te_3,  \quad e_3 \in E_1 \}.\\
\end{array}
\end{equation}

Imposing orthogonality conditions $IU \perp U, \quad JU \perp U,
\quad KU \perp U$, we obtain
$\; \omega^E|_{E_1}=\omega^E|_{E_2} \equiv 0$, from 1) and 2),
and
$ \; \omega^E(e,Te') + \omega^E(Te,e')=0, \quad \forall e, e' \in E_1$  from 3).

\vskip0.2cm

Viceversa,  given a pure real subspace $U=U^{F,T}$,
from (\ref{totally real conditions})  we obtain $IU \perp U, \quad JU \perp U, \quad KU \perp U$.
For any $X=h_1 \otimes e + h_2 \otimes Te$ and $Y=h_1
\otimes e' + h_2 \otimes Te'$, the get
\begin{equation} \label{metric in totally real subspace}
 g(X,Y)=\omega^E(e,Te')-
\omega^E(Te,e')=2 \omega^E(e,Te').
\end{equation}
Since  $U$ is  non degenerate, then   $\omega^E(E_1 \times E_2)$ is  non
degenerate hence $E_1 \cap E_2= \{0\}$.
\end{proof}



\subsection{Decomposition of a generic subspace} \label{section: Decomposition of a generic subspace}


Let $U \subset V$ be a subspace of the para-quaternionic Hermitian vector space
 $(V=H \otimes E, \tilde{Q}=\mathfrak{sl}(H) \otimes Id, g=\omega^H \otimes \omega^E)$.
 For any $A \in \widetilde{Q}$ we denote by $U_A$ the maximal  $A$-invariant subspace in $U$.

The following proposition, whose proof is straightforward by a procedure of successive decompositions,  expresses that, by using para-quaternionic, pure complex, weakly pure para-complex,
and real subspaces as building blocks, we can construct any subspace $U \subset V$.
\begin{prop} \label{decomposition of a generic subspace}
Let $U$ be a subspace in $V$ and $U_{0}$ be its  maximal para-quaternionic subspace. Then $U$ admits a direct sum decomposition of the form
\[
U= U_{0} \oplus {U'} \] with
\[
U'=U_{I_1} \oplus \ldots \oplus U_{I_p}\oplus U_{J_1} \oplus \ldots \oplus  U_{J_q} \oplus
U^R,
\]
 where
the $U_{I_i}, \, i=1,\dots,p,$ are pure $I_i$-complex subspaces,
the $U_{J_j}, \, i=1,\dots,q,$ are  $J_j$-pure weakly para-complex subspaces and
$U^R$ is real.

By using as building blocks pure para-complex subspaces instead of pure weakly para-complex, we necessarily need to use also pure nilpotent subspaces.

\end{prop}
As an example of the last statement  let think of a decomposable subspace $U=h\otimes F, \, h \in H, \, F \subset E$.

The decomposition of the proposition (\ref{decomposition of a generic subspace}) is clearly not unique even up to reordering of addends. The first reason depends  obviously on  the non uniqueness of the complement at each steps of the decomposition. Moreover the decomposition depends on the  chosen order of types of subspaces i.e. if we first consider pure complex subspaces and then pure weakly para-complex or the other way round. Not taking into account the metric,   we intend to further investigate if the different possible decompositions,
choosing the addends by decreasing dimension and  fixing the order of the decomposition,  are unique up to isomorphisms i.e. have addends of same type and dimension.

\begin{Acknow}
The author  would like to warmly thank Professor D. Alekseevsky for the fundamental help and
the advise given during the whole research.
\end{Acknow}


\end{document}